\documentclass[pagesize,pdftex]{scrartcl}
\usepackage{latexsym} 
\usepackage[intlimits]{amsmath}
\usepackage{amsthm}
\usepackage{amsfonts}
\usepackage{amssymb}
\usepackage{amsxtra}
\usepackage{amscd}
\usepackage{ifthen}
\usepackage{graphicx}
\usepackage[shortlabels]{enumitem}
\usepackage{mathrsfs}
\usepackage[pagebackref=true]{hyperref}
\usepackage{mathtools}

\hypersetup{
pdfauthor={Thomas Eiter, Mads Kyed},
pdftitle={Estimates of time-periodic fundamental solutions to the linearized Navier-Stokes equations},
breaklinks=true,
colorlinks=true,
linkcolor=blue,
citecolor=blue,
urlcolor=blue,
filecolor=blue,
}

\numberwithin{equation}{section} 
\setkomafont{title}{\normalfont}

\newcommand{\fundsolssralpha}{\varGamma^{\alpha}_{\text{\textnormal{\tiny{H}}}}}
\newcommand{\fundsolssrkrey}{\varGamma^{k,\rey}_{\text{\textnormal{\tiny{H}}}}}
\renewcommand{\Im}{\mathrm{Im}}

\newenvironment{pdeq}{ \left\{ \begin{aligned}}{\end{aligned}\right.}

\newcommand{\np}[1]{(#1)}
\newcommand{\nb}[1]{[#1]}
\newcommand{\bp}[1]{\big(#1\big)}
\newcommand{\bb}[1]{\big[#1\big]}
\newcommand{\Bp}[1]{\bigg(#1\bigg)}
\newcommand{\BBp}[1]{\Bigg(#1\Bigg)}
\newcommand{\Bb}[1]{\bigg[#1\bigg]}
\newcommand{\calj}{{\mathcal J}}
\newcommand{\calk}{{\mathcal K}}

\newcommand{\calt}{{\mathcal T}}

\newcommand{\R}{\mathbb{R}}
\newcommand{\C}{\mathbb{C}}
\newcommand{\Z}{\mathbb{Z}}

\newcommand{\N}{\mathbb{N}}
\DeclareMathOperator{\e}{e}
\newcommand{\B}{B}

\DeclareMathOperator{\Div}{div}

\DeclareMathOperator{\impart}{Im}

\newcommand{\ra}{\rightarrow}

\newcommand{\set}[1]{\ensuremath{\{#1\}}}

\newcommand{\closure}[2]{\overline{#1}^{#2}}
\newcommand{\quotientmap}{\pi}
\newcommand{\grp}{G}
\newcommand{\dualgrp}{\widehat{G}}

\newcommand{\torus}{{\mathbb T}}

\newcommand{\idmatrix}{I}
\newcommand{\Rn}{{\R^n}}

\newcommand{\ddz}{\frac{{\mathrm d}}{{\mathrm d}z}}
\newcommand{\grad}{\nabla}

\newcommand{\dx}{{\mathrm d}x}

\newcommand{\dsy}{{\mathrm d}(s,y)}
\newcommand{\dr}{{\mathrm d}r}

\newcommand{\ds}{{\mathrm d}s}
\newcommand{\dt}{{\mathrm d}t}

\newcommand{\dy}{{\mathrm d}y}

\newcommand{\dxi}{{\mathrm d}\xi}
\newcommand{\dyone}{{\mathrm d}y_1}

\newcommand{\SR}{\mathscr{S}}

\newcommand{\TDR}{\mathscr{S^\prime}}

\newcommand{\FT}{\mathscr{F}}
\newcommand{\iFT}{\mathscr{F}^{-1}}
\newcommand{\riesztrans}{\mathfrak{R}}
\newcommand{\Mmultiplier}{M}

\newcommand{\hankelone}[1]{H^{(1)}_{#1}}
\newcommand{\norm}[1]{\lVert#1\rVert}

\newcommand{\snorm}[1]{{\lvert #1 \rvert}}
\newcommand{\snorml}[1]{{\bigl\lvert #1 \big\rvert}}
\newcommand{\snormL}[1]{{\Bigl\lvert #1 \Big\rvert}}

\newcommand{\WSR}[2]{W^{#1,#2}}

\newcommand{\CR}[1]{C^{#1}}  
\newcommand{\LR}[1]{L^{#1}}

\newcommand{\LRloc}[1]{L^{#1}_{loc}} 

\newcommand{\CRci}{\CR \infty_0}
\newcommand{\CRc}[1]{\CR{#1}_0}
\newcommand{\wvel}{w}

\newcommand{\uvel}{u}

\newcommand{\upres}{\mathfrak{p}}

\newcommand{\fundsolvel}{\varGamma}

\newcommand{\fundsolpres}{\gamma}

\newcommand{\fundsolsteadystatevel}{\varGamma^{\textnormal{\tiny{}}}}
\newcommand{\fundsolsteadystatepres}{\gamma^{\textnormal{\tiny{}}}}

\newcommand{\fundsolstokesvel}{\varGamma^{\textnormal{\tiny{S}}}}

\newcommand{\fundsoloseenvel}{\varGamma^{\textnormal{\tiny{O}}}}

\newcommand{\fundsoltp}{\Phi}
\newcommand{\fundsoltpvel}{\varGamma^{\textnormal{\tiny{TP}}}}
\newcommand{\fundsoltppres}{\gamma^{\textnormal{\tiny{TP}}}}

\newcommand{\fundsolcompl}{\varGamma^\bot}

\newcommand{\fundsollaplace}{\varGamma_{{\textnormal{\tiny{L}}}}}
\newcommand{\tin}{\text{in }}
\newcommand{\tif}{\text{if }}

\newcommand{\half}{\frac{1}{2}}

\renewcommand{\epsilon}{\varepsilon}
\renewcommand{\phi}{\varphi}
\newcommand{\rey}{\lambda}

\newcommand{\tay}{\calt}
\newcommand{\per}{\tay}

\newcommand{\perf}{\frac{2\pi}{\tay}}

\newcommand{\bigo}{O}

\newcommand{\cutoff}{\chi}

\newcommand{\onedist}{1}
\newcommand{\change}[1]{}
\newcommand{\F}{F}
\newcommand{\newCCtr}[2][d]{
\newcounter{#2}\setcounter{#2}{0}
\expandafter\xdef\csname kyedtheconst#2\endcsname{#1}
}
\newcommand{\Cc}[2][nolabel]{
\stepcounter{#2}
\expandafter\ensuremath{\csname kyedtheconst#2\endcsname_{\arabic{#2}}}
\ifthenelse{\equal{#1}{nolabel}}
{}
{\expandafter\xdef\csname kyedconst#1\endcsname
{\expandafter\ensuremath{\csname kyedtheconst#2\endcsname_{\arabic{#2}}}}}
}
\newcommand{\Ccn}[2][nolabel]{
\expandafter\ensuremath{\csname kyedtheconst#2\endcsname}
\ifthenelse{\equal{#1}{nolabel}}
{}
{\expandafter\xdef\csname kyedconst#1\endcsname
{\expandafter\ensuremath{\csname kyedtheconst#2\endcsname}}}
}
\newcommand{\CcSetCtr}[2]{
\setcounter{#1}{#2}
}
\newcommand{\Cclast}[1]{
\expandafter\ensuremath{\csname kyedtheconst#1\endcsname_{\arabic{#1}}}
}
\newcommand{\Ccllast}[1]{
\addtocounter{#1}{-1}
\expandafter\ensuremath{\csname kyedtheconst#1\endcsname_{\arabic{#1}}}
\addtocounter{#1}{1}
}
\newcommand{\const}[1]{
\expandafter{\ifcsname kyedconst#1\endcsname
  \csname kyedconst#1\endcsname
\else
  \errmessage{Undefined Kyedconstant #1.}%
\fi}
}

\theoremstyle{plain}
\newtheorem{thm}{Theorem}[section]

\newtheorem{lem}[thm]{Lemma}

\theoremstyle{remark}
\newtheorem{rem}[thm]{Remark}

\begin{document}
\title{Estimates of time-periodic fundamental solutions to the linearized Navier-Stokes equations}

\author{
Thomas Eiter\\ 
Fachbereich Mathematik\\
Technische Universit\"at Darmstadt\\
Schlossgartenstr. 7, 64289 Darmstadt, Germany\\
Email: {\texttt{eiter@mathematik.tu-darmstadt.de}}
\and
Mads Kyed\\ 
Fachbereich Mathematik\\
Technische Universit\"at Darmstadt\\
Schlossgartenstr. 7, 64289 Darmstadt, Germany\\
Email: {\texttt{kyed@mathematik.tu-darmstadt.de}}
}

\date{\today}
\maketitle

\begin{abstract}
Fundamental solutions to the time-periodic Stokes and Oseen linearizations of the Navier-Stokes equations in dimension $n\geq 2$ are investigated. Integrability properties and pointwise estimates are established.
\end{abstract}


\newCCtr[C]{C}
\newCCtr[M]{M}
\newCCtr[\epsilon]{eps}
\CcSetCtr{eps}{-1}
\newCCtr[c]{c}
\let\oldproof\proof
\def\proof{\CcSetCtr{c}{-1}\oldproof}

\section{Introduction}

The concept of a time-periodic fundamental solution to the Stokes equations was introduced recently in \cite{Kyed_FundsolTPStokes2016}. 
In the following, we extend this concept to the time-periodic Oseen equations in arbitrary dimension $n\geq 2$. 
Moreover, we establish pointwise estimates and integrability properties for both the Stokes and Oseen fundamental solution. 
At the outset, the fundamental solutions are introduced as distributions on a Schwartz-Bruhat space.
The integrability properties enable us to identify these distributions as elements of an appropriate
$\LR{q}$ space. Consequently, convolutions with the fundamental solutions can be expressed in terms
of classical integrals. The pointwise estimates can then be used to investigate, for example, the asymptotic  behavior of time-periodic solutions to both the Stokes and Oseen equations.

The idea of investigating time-periodic problems in terms of fundamental solutions is new. Classically, Poincar\'e maps have been used instead.
Consider the Poincar\'e map that takes a state at time $0$ into the state at time $\per>0$ of a 
solution to the (linearized) Navier-Stokes initial-value problem. 
A fixed point of this mapping yields a $\per$-time-periodic solution. More precisely,
the $\per$-time-periodic solution is identified as the specific solution to the 
initial-value problem corresponding to the fixed point.
Working with such an indirect 
identification, one faces a number of limitations. It is, for example, difficult 
to derive information on the pointwise structure of the solution.
Instead, we propose to express a time-periodic solution in terms of a convolution integral 
with an appropriate fundamental solution. In order to exploit such a direct representation formula, integrability properties and pointwise 
estimates of the fundamental solution are needed.

In the following, we investigate linearizations of the time-periodic Navier-Stokes equations in $\R^n$ with $n\geq 2$.
The time-period $\per>0$ remains fixed.
Let $\rey\in\R$. The linearization 
\begin{align}\label{tplinns}
\begin{pdeq}
&\partial_t\uvel-\Delta\uvel + \rey\partial_{x_1}\uvel +\grad\upres = f && \tin\R\times\Rn, \\
&\Div\uvel =0 && \tin\R\times\Rn,\\
&\uvel(t,x) = \uvel(t+\per,x)  
\end{pdeq}
\end{align} 
is referred to as the time-periodic Stokes system if $\rey=0$, and as the time-periodic Oseen system if $\rey\neq 0$.  
Here $\uvel:\R\times\Rn\ra\Rn$ and $\upres:\R\times\Rn\ra\R$ denote the Eulerian velocity field and pressure term, respectively.  
Data $f:\R\times\Rn\ra\Rn$ with the same period, that is, $f(t,x) = f(t+\per,x)$, are considered. 
Moreover, $t\in\R$ and $x\in\R^n$ denote the time and spatial variable, respectively.

In order to define a fundamental solution to \eqref{tplinns}, we employ the approach from \cite{Kyed_FundsolTPStokes2016}
and reformulate \eqref{tplinns} as a system of partial differential equations 
on the locally compact abelian group $\grp\coloneqq\R/\per\Z \times \R^n$.
More specifically, we exploit that $\per$-time-periodic functions can naturally be identified with 
functions on the 
torus group $\torus\coloneqq\R/\per\Z$ in the time variable $t$.
In the setting of the Schwartz-Bruhat space $\SR(\grp)$ and corresponding space of tempered distributions
$\TDR(\grp)$, we can then define a fundamental solution $\fundsoltp$ to \eqref{tplinns} as a 
tensor-field
\begin{align}\label{tpFundsolForm}
\fundsoltp\coloneqq
\begin{pmatrix}
\fundsoltpvel_{11} & \ldots  & \fundsoltpvel_{1n} \\
\vdots & \ddots & \vdots\\
\fundsoltpvel_{n1} & \ldots  & \fundsoltpvel_{nn} \\
\fundsoltppres_{1} & \ldots  & \fundsoltppres_{n} 
\end{pmatrix}
\in \TDR(\grp)^{(n+1)\times n}
\end{align}
that satisfies\footnote{We make use of the Einstein summation convention and
implicitly sum over all repeated indices.}
\begin{align}\label{tpFundsolEq}
\begin{pdeq}
&\partial_t\fundsoltpvel_{ij}-\Delta\fundsoltpvel_{ij} + \rey\partial_{x_1}\fundsoltpvel_{ij} + \partial_i \fundsoltppres_j = \delta_{ij}\delta_\grp,  \\
&\partial_i\fundsoltpvel_{ij} =0
\end{pdeq}
\end{align} 
in the sense of $\TDR(\grp)$-distributions. 
Here, $\delta_{ij}$ and $\delta_\grp$ denote the Kronecker delta and delta distribution, respectively.
A solution to the time-periodic system \eqref{tplinns} is then given by 
\begin{align}\label{intro_ConvolutionWithTPFundsol}
\begin{pmatrix}
\uvel \\
\upres
\end{pmatrix} \coloneqq \fundsoltp * f, 
\end{align}
where the component-wise convolution is taken over the group $\grp$. 

In the following, we shall identify a fundamental solution $\fundsoltp$ to \eqref{tplinns} as the sum of a fundamental solution to 
the corresponding steady-state system
\begin{align}\label{SteadyStateFundEq}
\begin{pdeq}
&-\Delta\fundsolsteadystatevel_{ij} +\rey\partial_{x_1}\fundsolsteadystatevel_{ij}+ \partial_i \fundsolsteadystatepres_j = \delta_{ij}\delta_\Rn,  \\
&\partial_i\fundsolsteadystatevel_{ij} =0, 
\end{pdeq}
\end{align} 
and a second part, which we refer to as the \emph{purely periodic} part. 
Recall that in the Stokes case ($\rey=0$) a fundamental solution to \eqref{SteadyStateFundEq} is given by (see for example \cite[IV.2]{GaldiBookNew}) 
\begin{align}\label{SteadyStateStokesFundEq}
\begin{aligned}
&\fundsolstokesvel_{ij}(x)\coloneqq 
\begin{pdeq}
&\frac{1}{2\omega_n}\Bp{\delta_{ij}\log\bp{\snorm{x}^{-1}}+\frac{x_ix_j}{\snorm{x}^2} } && \text{if }n=2, \\
&\frac{1}{2\omega_n}\Bp{\delta_{ij}\frac{1}{n-2}\snorm{x}^{2-n}+\frac{x_ix_j}{\snorm{x}^n}} && \text{if }n\geq 3. 
\end{pdeq}
\end{aligned}
\end{align}
Here, $\omega_n$ denotes the surface area of the $(n-1)$-dimensional unit sphere in $\Rn$. 
In the Oseen case ($\rey\neq 0$), a fundamental solution to \eqref{SteadyStateFundEq} is given by (see for example \cite[VII.3]{GaldiBookNew}) 
\begin{align}\label{SteadyStateOseenFundEq}
&\fundsoloseenvel_{ij}(x) \coloneqq \frac{1}{\rey} \left( \delta_{ij} \Delta - \partial_{x_i} \partial_{x_j} \right) \int_0^{x_1} \left[ \fundsollaplace(y_1, x_2, \ldots, x_n) - \Psi(y_1,x_2, \ldots, x_n) \right] \dyone,
\end{align} 
where
\begin{align*}
\fundsollaplace(x)\coloneqq
\begin{pdeq}
&-\frac{1}{2\pi} \log\snorm{x} &&\tif n=2,\\
&\frac{1}{(n-2)\omega_n} \snorm{x}^{2-n} && \tif n>2,
\end{pdeq}
\end{align*}
is the fundamental solution to the Laplace equation $\Delta\fundsollaplace = \delta_{\R^n}$ in $\R^n$ and
\[
\Psi(x) \coloneqq - \frac{1}{2\pi} \left( \frac{\rey}{4\pi \snorm{x}} \right)^{\frac{n-2}{2}} K_{\frac{n-2}{2}}\left(\frac{\rey}{2}\snorm{x}\right) \e^{-\frac{\rey}{2}x_1}.
\]
Here $K_\nu$ denotes the modified Bessel function of the second kind.
In both the Stokes and Oseen case the pressure term in the fundamental solution is given by
\begin{align}\label{SteadyStateFundsolPressure}
\begin{aligned}
&\fundsolpres_{i}(x)\coloneqq \frac{1}{\omega_n}\frac{x_i}{\snorm{x}^n}.
\end{aligned}
\end{align}
In order to identify the second part of $\fundsoltp$, that is, the purely periodic part, we
utilize the Fourier transform $\FT_\grp$ on the group $\grp$. The fact that 
$\FT_\grp:\TDR(\grp)\ra\TDR(\dualgrp)$ is a homeomorphism allows us to express the purely periodic part 
in terms of a Fourier multiplier defined on the dual group $\dualgrp=\Z\times\R^n$.
Our main theorem reads:

\begin{thm}\label{MainThm}
Let $n\geq 2$ and $\lambda \in \R$. Put
\begin{align*}
\fundsolsteadystatevel\coloneqq 
\begin{pdeq}
&\fundsolstokesvel && \tif \rey=0 \quad \text{(Stokes case)},\\
&\fundsoloseenvel  && \tif \rey\neq0 \quad \text{(Oseen case)}.
\end{pdeq}
\end{align*}
Then the elements of $\TDR(\grp)$ given by
\begin{align}
&\fundsoltpvel \coloneqq \fundsolsteadystatevel\otimes \onedist_{\torus} + \fundsolcompl,\label{MainThm_decompVel}\\
&\fundsoltppres \coloneqq \fundsolsteadystatepres \otimes \delta_{\torus},\label{MainThm_decompPres}
\end{align}
with
\begin{align}\label{MainThm_deffundsolcompl}
\fundsolcompl \coloneqq \iFT_\grp\Bb{ \frac{1-\delta_\Z(k)}{\snorm{\xi}^2 + i(\perf k + \rey \xi_1)}\,\Bp{\idmatrix - \frac{\xi\otimes\xi}{\snorm{\xi}^2}}}\in\TDR(\grp)^{n\times n}
\end{align}
define a fundamental solution $\fundsoltp\in\TDR(\grp)^{(n+1)\times n}$ to \eqref{tpFundsolEq} of the form \eqref{tpFundsolForm} satisfying
\begin{align}
&\forall q\in\Bp{1,\frac{n+2}{n}}:\quad \fundsolcompl\in\LR{q}(\grp)^{n\times n},\label{MainThm_ComplSummability}\\
&\forall q\in\bigg[1,\frac{n+2}{n+1}\bigg):\quad \partial_j\fundsolcompl\in\LR{q}(\grp)^{n\times n}\quad(j=1,\ldots,n),\label{MainThm_ComplSummabilityGradient}\\
&\forall \alpha \in \N_0^n \ \forall r\in [1,\infty)\ \forall\epsilon>0\ \exists\Cc[MainThm_ComplPointwiseEstConstant]{C}>0\ \forall \snorm{x}\geq \epsilon:\  \norm{D_x^\alpha \fundsolcompl(\cdot,x)}_{\LR{r}(\torus)} \leq \frac{\const{MainThm_ComplPointwiseEstConstant}}{\snorm{x}^{n+\snorm{\alpha}}},\label{MainThm_ComplPointwiseEst}\\
&\forall q\in(1,\infty)\ \exists \Cc[MainThm_ConvlFundsolcomplLqEstConstant]{C}>0\ \forall \F\in\SR(\grp)^n:\  \norm{\fundsolcompl*\F}_{\WSR{1,2}{q}(\grp)} \leq \const{MainThm_ConvlFundsolcomplLqEstConstant}\,\norm{\F}_{\LR{q}(\grp)},  \label{MainThm_ConvlFundsolcomplLqEst}
\end{align}
where $\delta_\torus\in\TDR(\torus)$ denotes the delta distribution and $\onedist_\torus\in\TDR(\torus)$ the constant $1$ distribution.\footnote{We thank Tomoyuki Nakatsuka for pointing out that our proof of \eqref{MainThm_ComplPointwiseEst} also works for derivatives of arbitrary order.}
\end{thm}

\begin{rem}
Consider data $f\in\CRc{}(\grp)^n$. The integrability of $\fundsolcompl$ obtained in \eqref{MainThm_ComplSummability} implies, in combination with well-known integrability properties of the steady-state fundamental solutions  \eqref{SteadyStateStokesFundEq} and \eqref{SteadyStateOseenFundEq}, that 
the solution $\uvel$ to \eqref{tplinns} given 
by the convolution \eqref{intro_ConvolutionWithTPFundsol} can be written in terms of classical integrals as ($i=1,\ldots,n$)
\begin{align*}
\uvel_i(t,x)&= \bp{\fundsolsteadystatevel\otimes \onedist_{\torus}}_{ij}*f_j(t,x) + \fundsolcompl_{ij}*f_j(t,x)\\
&= \int_{\R^n} \fundsolsteadystatevel_{ij}(x-y)\Bp{\int_{\torus}f_j(s,y)\,\ds}\,\dy
+\int_\grp  \fundsolcompl_{ij}(t-s,x-y) f_j(s,y)\,\dsy\\
&=: \uvel_i^s(x) + \uvel_i^p(t,x).
\end{align*}
Observe that $\uvel^s$ is a solution to a steady-state Stokes ($\rey=0$) or Oseen ($\rey\neq 0$) problem.
Consequently, the pointwise asymptotic structure at spatial infinity of $\uvel^s$ is well known; see for example 
\cite[Theorem V.3.2 and VII.6.2]{GaldiBookNew}. From 
\eqref{MainThm_ComplSummability} and \eqref{MainThm_ComplPointwiseEst}
it follows that $\uvel^p(t,x)=\bigo(\snorm{x}^{-n})$, which means that the decay rate of $\uvel^p(t,x)$ as $\snorm{x}\ra\infty$ is actually \emph{faster} than 
the decay rate of the leading term in the asymptotic expansion of $\uvel^s(x)$. In other words,
the leading term in an asymptotic expansion of $\uvel$ coincides with the (known) leading term in the expansion of $\uvel^s$. 
This direct consequence of Theorem \ref{MainThm} is by no means trivial.    
\end{rem}

\section{Preliminaries}

We denote by $\B_R\coloneqq\B_R(0)$ balls in $\Rn$ centered at $0$. Moreover, we let $\B_{R,r}\coloneqq\B_R\setminus\overline{\B_r}$ and $\B^R\coloneqq\Rn\setminus\overline{\B_R}$.

For a sufficiently regular function $u:\R\times\Rn\ra\R$, we put $\partial_i u\coloneqq\partial_{x_i} u$.
The differential operators $\Delta$, $\grad$ and $\Div$ only act in the spatial variables.
 
We let $\grp$ denote the group $\grp\coloneqq\torus \times \R^n$, with $\torus$ denoting the torus group $\torus\coloneqq\R/\per\Z$. 
$\grp$ is equipped with the quotient topology and differentiable structure inherited from $\R\times\R^n$ via the quotient mapping
$\quotientmap:\R\times\R^n\ra\grp$, $\quotientmap(t,x)\coloneqq\bp{[t],x}$.
Clearly, $\grp$ is a locally compact abelian group with Haar measure given by the product of the (normalized) Haar measure $\dt$ on $\torus$ and the Lebesgue measure $\dx$ on $\Rn$. 
We implicitly identify $\torus$ with the interval $[0,\per)$, whence the (normalized) Haar measure on $\torus$ is determined by
\begin{align*}
\forall f\in\CR{}(\torus):\quad \int_\torus f\,\dt \coloneqq \frac{1}{\per}\int_0^\per f(t)\,\dt.
\end{align*}
We identify the dual group $\dualgrp$ with $\Z\times\Rn$ and denote points in $\dualgrp$ by $(k,\xi)$. 

We denote the Schwartz-Bruhat space of generalized Schwartz functions by $\SR(\grp)$; see \cite{Bruhat61}. By 
$\TDR(\grp)$ we denote the corresponding space of tempered distributions. The Fourier transformation on $\grp$ and its inverse take the form 
\begin{align*}
&\FT_\grp:\SR(\grp)\ra\SR(\dualgrp),\quad \FT_\grp\nb{\uvel}(k,\xi)\coloneqq
\int_\torus\int_{\R^n} \uvel(t,x)\,\e^{-ix\cdot\xi-ik\perf t}\,\dx\dt,\\
&\iFT_\grp:\SR(\dualgrp)\ra\SR(\grp),\quad \iFT\nb{\wvel}(t,x)\coloneqq
\sum_{k\in\Z}\,\int_{\R^n} \wvel(k,\xi)\,\e^{ix\cdot\xi+ik\perf t}\,\dxi,
\end{align*}
respectively, provided the Lebesgue measure $\dxi$ is normalized appropriately. By duality, $\FT_\grp$ extends to a homeomorphism $\FT_\grp:\TDR(\grp)\ra\TDR(\dualgrp)$. Observe that $\FT_\grp=\FT_\Rn\circ\FT_\torus$.

We denote the Dirac delta distribution on $\Rn$, $\torus$ and $\Z$ by $\delta_\Rn$, $\delta_\torus$ and $\delta_\Z$ , respectively.
Observe that $\delta_\Z$ is a function with $\delta_\Z(k)=1$ if $k=0$ and $\delta_\Z(k)=0$ otherwise. Also note that 
$\FT_\torus\nb{1_\torus}=\delta_\Z$.

Given a tensor $\Gamma\in\TDR(\grp)^{n\times m}$, we define the convolution of $\Gamma$ with vector field $f\in\SR(\grp)^m$ as 
the vector field $\Gamma * f \in\TDR(\grp)^n$ with $\nb{\Gamma*f}_{i} \coloneqq \Gamma_{ij}*f_j$.

The $\LR{q}(\grp)$-spaces with norm $\norm{\cdot}_q$ are defined in the usual way via the Haar measure $\dx\dt$ on $\grp$. We further introduce 
the Sobolev space
\begin{align*}
&\WSR{1,2}{q}(\grp) \coloneqq \closure{\CRci(\grp)}{\norm{\cdot}_{1,2,q}},\quad 
\norm{f}_{1,2,q}\coloneqq\Bp{
\norm{\partial_t f}_{q}^q +
\sum_{\snorm{\alpha}\leq 2} \norm{\partial_x^\alpha f}_{q}^q  
}^{\frac{1}{q}},
\end{align*}
where $\CRci(\grp)$ denotes the space of smooth functions of compact support on $\grp$.

We emphasize at this point that a framework based on $\grp$ is a natural setting for the time-period Stokes and Oseen equations. It it easy to see 
that lifting by the restriction $\quotientmap_{|\R^n\times[0,\per)}$ of the quotient mapping provides us with an equivalence between the time-periodic 
linearization \eqref{tplinns} and the system
\begin{align*}
\begin{pdeq}
&\partial_t\uvel-\Delta\uvel + \rey \partial_1 \uvel + \grad\upres = f && \tin\grp, \\
&\Div\uvel =0 && \tin\grp.
\end{pdeq}
\end{align*}
An immediate advantage obtained by writing the time-periodic Stokes or Oseen problem as system of equations on $\grp$ is the ability to then apply the Fourier transform $\FT_\grp$ and re-write the problem in terms of Fourier symbols. We shall take advantage of this possibility in the proof of the main theorem below.  

Constants in capital letters in the proofs and theorems are global, while constants in small letters are local to the proof in which they appear. Unless otherwise stated, constants are positive.

\section{Proof of the main theorem}
A large amount of the proof of Theorem \ref{MainThm} is based on pointwise estimates of the functions
\begin{align}\label{defofFundsolSSRk}
\fundsolssralpha \colon \R^n\setminus\set{0}\ra\C,\quad \fundsolssralpha(x) \coloneqq \frac{i}{4} \BBp{\frac{\sqrt{-\alpha}}{2\pi \snorm{x}}\,}^{\frac{n-2}{2}}\hankelone{\frac{n-2}{2}}
\Bp{\sqrt{-\alpha}\cdot \snorm{x}}
\end{align}
and
\begin{align}
\fundsolssrkrey \colon \R^n\setminus\set{0}\ra\C,\quad \fundsolssrkrey \coloneqq \fundsolssralpha (x) \cdot \e^{\frac{\rey}{2}x_1}
\end{align}
with $\alpha \coloneqq (\rey/2)^2+i\perf k$ and $k \in \Z$.
Here $\hankelone{\nu}$ denotes the Hankel function of the first kind and $\sqrt{z}$ the square root of $z$ with \emph{nonnegative} imaginary part.
It is well known, see for example \cite[Chapter 5.8]{Stakgold_BrdValueBook}, that $\fundsolssralpha$ is a fundamental solution in $\TDR(\R^n)$ to the Helmholtz equation
\begin{align}\label{HelmholtzEqFundsolEq}
\bp{-\Delta + \alpha} \fundsolssralpha = \delta_{\R^n}
\end{align}
when $\Im (\alpha) \neq 0$, which is the case if $k \neq 0$.
In order to analyze $\fundsolssrkrey$, we first recall the following properties of Hankel functions:
\begin{lem}\label{estHankelFktlem}
Hankel functions are analytic in $\C\setminus\set{0}$ with
\begin{align}\label{diffOfHankelFormula}
\forall\nu\in\C\ \forall z\in\C\setminus\set{0}:\quad \ddz\hankelone{\nu} (z) = \hankelone{\nu-1}(z) - \frac{\nu}{z}\,\hankelone{\nu}(z).
\end{align}
The Hankel functions satisfy the following estimates:
\begin{align}
&\forall\nu\in\C\ \forall\epsilon>0\ \exists \Cc[estHankelFktInftyConst]{C}>0\ \forall \snorm{z}\geq \epsilon:&&
\snorml{\hankelone{\nu} \np{z}}\leq \const{estHankelFktInftyConst}\, \snorm{z}^{-\frac{1}{2}}\, \e^{-\impart z},
\label{estHankelFktAtInfty}\\
&\forall\nu\in\R_+\ \forall R>0\ \exists \Cc[estHankelFktZeroConst]{C}>0\ \forall \snorm{z}\leq R:&&
\snorml{\hankelone{\nu} \np{z}}\leq 
\const{estHankelFktZeroConst}\, \snorm{z}^{-\nu},\label{estHankelFktAtZero}\\
&\forall 0\leq R<1\ \exists \Cc[estHankelzeroFktZeroConst]{C}>0\ \forall \snorm{z}\leq R: && 
\snorml{\hankelone{0} \np{z}}\leq 
\const{estHankelzeroFktZeroConst}\, \snorml{\log\np{\snorm{z}}}. \label{estHankelzeroFktAtZero}
\end{align}
\end{lem}
\begin{proof}
The recurrence relation \eqref{diffOfHankelFormula} is a well-know property of various Bessel functions; see for example
\cite[9.1.27]{AbramowitzStegunHandbook}.
We refer to \cite[9.2.3]{AbramowitzStegunHandbook} for the asymptotic behavior \eqref{estHankelFktAtInfty} of $\hankelone{\nu}(z)$ as $z\ra \infty$.
See \cite[9.1.9 and 9.1.8]{AbramowitzStegunHandbook} for the asymptotic behavior \eqref{estHankelFktAtZero} and \eqref{estHankelzeroFktAtZero} 
of $\hankelone{\nu}(z)$ as $z\ra 0$.
\end{proof}

At first we want to show that $\fundsolssrkrey$ is a fundamental solution in $\TDR(\R^n)$ to the equation
\begin{align}\label{PDEfundsolssrkrey}
\bp{-\Delta +\rey \partial_1+ i\perf k} \fundsolssrkrey = \delta_{\R^n}.
\end{align}
For this purpose, we want to use the following technical lemma, which will also be important for the derivation of the pointwise estimates claimed in Theorem \ref{MainThm}.
\begin{lem}\label{LemEstImalpha}
There exists a constant $\Cc[Imalpha]{C}=\const{Imalpha}(\rey,\per)>0$ such that 
\begin{align}\label{EstImalpha}
\frac{\snorm{\rey}}{2} - \Im( \sqrt{-\alpha}) \leq -\const{Imalpha} \snorm{k}^\frac{1}{2}
\end{align}
for all $k \in \Z \setminus \set{0}$.
\end{lem}
\begin{proof}
For $\rey=0$ the statement follows directly. If we assume $\lambda \neq 0$, it holds
\[
\sqrt{-\alpha}=\BBp{\Bp{\frac{\rey}{2}}^4 +\Bp{\perf k}^2 }^\frac{1}{4} \exp \Bp{\frac{i}{2} \Bp{\pi +\arctan \bp{ \perf k \cdot 4\rey^{-2}}}}.
\]
Thus we obtain
\begin{align*}
\Im(\sqrt{-\alpha}) & =\BBp{\Bp{\frac{\rey}{2}}^4 +\Bp{\perf k}^2 }^\frac{1}{4} \sin \BBp{\frac{1}{2} \Bp{\pi +\arctan \bp{ \perf k \cdot 4\rey^{-2}}}} \\
& =\BBp{\Bp{\frac{\rey}{2}}^4 +\Bp{\perf k}^2 }^\frac{1}{4} \cos \Bp{\frac{1}{2} \arctan \bp{ \perf k \cdot 4\rey^{-2}}} \\
& = \BBp{\Bp{\frac{\rey}{2}}^4 +\Bp{\perf k}^2 }^\frac{1}{4} \frac{1}{\sqrt{2}}\Bp{1+\cos \Bp{ \arctan \bp{ \perf k \cdot 4\rey^{-2}}}}^\frac{1}{2} \\
& = \BBp{\Bp{\frac{\rey}{2}}^4 +\Bp{\perf k}^2 }^\frac{1}{4} \frac{1}{\sqrt{2}}\Bp{1+\Bp{1+\bp{\perf k \cdot 4\rey^{-2}}^2}^{-\frac{1}{2}}}^\frac{1}{2} \\
& = \frac{\snorm{\rey}}{2} \frac{1}{\sqrt{2}} \Bp{\Bp{1+\frac{\bp{\perf k}^2}{\bp{\rey/2}^4}}^\frac{1}{2} + 1 }^\frac{1}{2}.
\end{align*}
Consequently, it holds
\[
\frac{\snorm{\rey}}{2} - \Im( \sqrt{-\alpha}) <0
\]
and
\[
\lim_{\snorm{k} \to \infty}  \frac{\frac{\snorm{\rey}}{2} - \Im( \sqrt{-\alpha})}{\snorm{k}^\frac{1}{2}} = - \sqrt{\frac{\perf}{2}}.
\]
This implies the assertion and finishes the proof.
\end{proof}

With the help of the previous two lemmata, we can now verify that $\fundsolssrkrey$ is an element of $\TDR(\R^n)$ and satisfies equation \eqref{PDEfundsolssrkrey}.

\begin{lem}\label{lemfunsolssrkreyDistr}
For all $k \in \Z \setminus \set{0}$, the function $\fundsolssrkrey$ is a fundamental solution in $\TDR(\R^n)$ to equation \eqref{PDEfundsolssrkrey}, and it holds
\begin{align}\label{fundsolssrkreyMultiplier}
\fundsolssrkrey = \iFT_{\R^n} \Bb{\frac{1}{\snorm{\xi}^2 + i \perf k + i \rey \xi_1}}.
\end{align}
\end{lem}
\begin{proof}
We want to utilize the estimates from Lemma \ref{estHankelFktlem}. For $x \in \R^n$, $x \neq 0$, with $\snorm{x} \leq \half$, estimate \eqref{estHankelFktAtZero} implies
\[
\snorm{\fundsolssrkrey(x)} \leq \Cc{c} \snorm{x}^{-n+2} \e^{\frac{\rey}{2}x_1}
\]
if $n > 2$, and estimate \eqref{estHankelzeroFktAtZero} yields
\[
\snorm{\fundsolssrkrey(x)} \leq \Cc{c} \snorm{\log ( \snorm{\alpha}^\frac{1}{2} \snorm{x})}\, \e^{\frac{\rey}{2}x_1}
\]
in the case $n=2$. For $\snorm{x}>\half$, inequality \eqref{estHankelFktAtInfty} in combination with Lemma \ref{LemEstImalpha} implies
\begin{align*}
\snorm{\fundsolssrkrey(x)} 
& \leq \const{estHankelFktInftyConst} \snorm{\alpha}^{\frac{n-3}{4}} \snorm{x}^\frac{1-n}{2} \, \exp \Bb{-\Im \bp{\sqrt{-\alpha} \snorm{x}} + \frac{\rey}{2}x_1} \\
& \leq \const{estHankelFktInftyConst} \snorm{\alpha}^{\frac{n-3}{4}} \snorm{x}^\frac{1-n}{2} \, \exp \Bb{\Bp{-\Im \bp{\sqrt{-\alpha}} + \frac{\snorm{\rey}}{2}}\snorm{x}} \\
& \leq \const{estHankelFktInftyConst} \snorm{\alpha}^{\frac{n-3}{4}} \snorm{x}^\frac{1-n}{2} \, \exp \bb{-\const{Imalpha} \snorm{k}^{\frac{1}{2}} \snorm{x}}.
\end{align*}
In total, we have seen $\fundsolssrkrey \in \LR{1}(\R^n) \subseteq \TDR(\R^n)$. Moreover, by utilizing equation \eqref{HelmholtzEqFundsolEq}, a straightforward computation leads to
\[
\bp{-\Delta +\rey \partial_1+ i \perf k} \fundsolssrkrey = \bp{- \Delta \fundsolssralpha + \alpha \fundsolssralpha } \,  \e^{\frac{\rey}{2} x_1} = \delta_{\R^n}.
\]
Applying the Fourier transformation to this equality, we directly obtain identity \eqref{fundsolssrkreyMultiplier} since $k \neq 0$.
\end{proof}

The next lemma will be the key for the derivation of the pointwise estimate
\eqref{MainThm_ComplPointwiseEst}
from Theorem \ref{MainThm}. We proceed analogously to the proof of \cite[Lemma 3.2]{Kyed_FundsolTPStokes2016}.

\begin{lem}\label{pointwiseEstLem}
If we define
\begin{align}\label{defoffundsollaplaceCONVLfundsolssrk}
\fundsollaplace*\fundsolssrkrey (x) \coloneqq \int_{\R^n} \fundsollaplace(x-y)\,\fundsolssrkrey(y)\,\dy,
\end{align}
then for all $\alpha \in \N_0^n$ and all $\epsilon>0$ there exists a constant $\Cc[pointwiseEstLemC]{C}>0$ such that for all $x \in \R^n$ with $\snorm{x} \geq \epsilon$ it holds
\begin{align}
\snorml{\partial_i\partial_j D_x^\alpha \nb{\fundsollaplace*\fundsolssrkrey}(x)} & \leq \const{pointwiseEstLemC}\,\snorm{k}^{-1}\,\snorm{x}^{-n-\snorm{\alpha}},\label{pointwiseEstLemEst1} 
\end{align}
for all $i,j\in \{1, \ldots, n\}$.
\end{lem}
\begin{proof}
In the proof of Lemma \ref{lemfunsolssrkreyDistr}, we have seen that $\fundsolssrkrey(y)$ decays exponentially as $\snorm{y}\ra\infty$. 
Therefore, one may verify directly from the pointwise definitions of $\fundsolssrkrey$ and $\fundsollaplace$ that the convolution integral in \eqref{defoffundsollaplaceCONVLfundsolssrk}
exists for all $x\in\R^n\setminus\set{0}$ and belongs to $\LRloc{1}(\R^n)$.
One may further verify that also the second order derivatives of $\fundsollaplace*\fundsolssrkrey$ are given pointwise by convolution integrals:
\begin{align}\label{fundsollaplaceCONVLfundsolssrk_secondorderderivatives}
\partial_i\partial_j \nb{\fundsollaplace*\fundsolssrkrey} (x) = \int_{\R^n} \partial_i\fundsollaplace(x-y)\,\partial_j\fundsolssrkrey(y)\,\dy.
\end{align}
Now fix $\epsilon>0$ and consider some $x\in\R^n$ with $\snorm{x}\geq \epsilon$. Put $R\coloneqq\frac{\snorm{x}}{2}$.  
Let $\cutoff\in\CRci(\R;\R)$ be a ``cut-off'' function with
\begin{align*}
\cutoff(r)=
\begin{pdeq}
&0 && \text{when }{0\leq \snorm{r}\leq \half},\\
&1 && \text{when }{1\leq \snorm{r}\leq 3},\\
&0 && \text{when }{4\leq \snorm{r}}.
\end{pdeq}
\end{align*}
Define $\cutoff_R:\R^n\ra\R$ by $\cutoff_R(y)\coloneqq\cutoff\bp{R^{-1}\snorm{y}}$. We decompose the integral in \eqref{fundsollaplaceCONVLfundsolssrk_secondorderderivatives} as
\begin{align*}
\partial_i\partial_j \nb{\fundsollaplace*\fundsolssrkrey} (x) 
&=\int_{B_{R}} \partial_i\fundsollaplace(x-y)\,\partial_j\fundsolssrkrey(y)\,\bp{1-\cutoff_R(y)}\,\dy\\ 
&\quad +\int_{B^{3R}} \partial_i\fundsollaplace(x-y)\,\partial_j\fundsolssrkrey(y)\,\bp{1-\cutoff_R(y)}\,\dy\\
&\quad + \int_{B_{4R,R/2}} \partial_i\fundsollaplace(x-y)\,\partial_j\fundsolssrkrey(y)\,\cutoff_R(y)\,\dy\\ 
&=: I_1(x) + I_2(x) + I_3(x).
\end{align*}
This also yields the decomposition
\begin{equation}\label{pointwiseEstLemDecomposition}
\partial_i\partial_j D_x^\alpha \nb{\fundsollaplace*\fundsolssrkrey}(x) = D_x^\alpha I_1(x) + D_x^\alpha I_1(x) + D_x^\alpha I_2(x).
\end{equation}
Recalling the definition \eqref{defofFundsolSSRk} of $\fundsolssrkrey$ as well as property \eqref{diffOfHankelFormula}
and estimate \eqref{estHankelFktAtInfty} of the Hankel function, we can estimate for $\snorm{y}\geq R/2$:  
\begin{align*}
\snorml{\partial_j \fundsolssrkrey(y)} 
\hspace{-40pt}& \hspace{40pt} = \snorml{\partial_j \fundsolssralpha(y) \, \e^{\frac{\rey}{2}y_1} + \fundsolssralpha(y)  \, \frac{\delta_{1j} \rey}{2} \e^{\frac{\rey}{2}y_1}} \\
&\leq \Cc{c} \e^{\frac{\snorm{\rey}}{2}\snorm{y}} \snorm{\alpha}^{\frac{n-2}{4}}\Bp{
\snormL{\partial_j\Bb{\snorm{y}^{\frac{2-n}{2}}}\, \hankelone{\frac{n-2}{2}} \bp{\sqrt{-\alpha} \cdot \snorm{y}}}
+\snormL{{\snorm{y}^{\frac{2-n}{2}}}\, \partial_j\Bb{\hankelone{\frac{n-2}{2}} \bp{\sqrt{-\alpha} \cdot \snorm{y}}}} \\
& \hspace{250pt} + \snormL{\snorm{y}^{\frac{2-n}{2}} \hankelone{\frac{n-2}{2}} \bp{\sqrt{-\alpha} \cdot \snorm{y}} \rey}
} \\
&\leq \Cc{c} \Bp{\snorm{\alpha}^{\frac{n}{4}-\frac{3}{4}}\, \snorm{y}^{-\frac{n}{2}-\half} + \snorm{\alpha}^{\frac{n}{4}-\frac{1}{4}}\, \snorm{y}^{-\frac{n}{2}+\half}}
\e^{-\Im (\sqrt{-\alpha})\snorm{y} + \frac{\snorm{\rey}}{2}\snorm{y}}\\
&\leq \Cc{c} \snorm{k}^{-1}\,\snorm{y}^{-(n+1)} \Bp{\bp{\snorm{k}^\half \snorm{y}}^{\frac{n}{2}+\half} 
+ \bp{\snorm{k}^\half \snorm{y}}^{\frac{n}{2}+\frac{3}{2}}}
\e^{-\const{Imalpha}\snorm{k}^\frac{1}{2} \snorm{y}}\\
&\leq \Cc{c} \snorm{k}^{-1}\,\snorm{y}^{1-n},
\end{align*}
where we employed \eqref{EstImalpha} and the inequality $\Cc{c} \snorm{k} \leq \snorm{\alpha} \leq \Cc{c} \snorm{k}$.
In the same way, we may deduce
\begin{equation}\label{EstDerivativeGammaH}
\snorml{D_x^\alpha \fundsolssrkrey(y)} \leq \Cc{c} \snorm{k}^{-1}\,\snorm{y}^{2-n-\snorm{\alpha}}
\end{equation}
for all $\snorm{y} \geq R/2$ and all $\alpha \in \N_0^n$.
To estimate $D_x^\alpha I_1$, we use integration by parts and employ polar coordinates to deduce
\begin{align*}
\snorml{D_x^\alpha I_1(x)} 
&\leq \Cc{c} 
\int_{B_R}\snorml{\partial_j\partial_i D_x^\alpha \fundsollaplace(x-y)}\, \snorml{\fundsolssrkrey(y)} +
\snorml{\partial_i D_x^\alpha \fundsollaplace(x-y)}\, \snorml{\fundsolssrkrey(y)}\, R^{-1}\,\dy \\
&\leq \Cc{c} \int_{B_R} R^{-n-\snorm{\alpha}}\, \snorml{\fundsolssrkrey(y)}\,\dy\\
&\leq \Cc{c} \int_{B_R} R^{-n-\snorm{\alpha}}\, \snorm{\alpha}^{\frac{n-2}{4}}\, \snorm{y}^{\frac{2-n}{2}}\, \snormL{\hankelone{\frac{n-2}{2}}\bp{\sqrt{-\alpha}\cdot \snorm{y}}} \, \e^{\frac{\snorm{\rey}}{2}\snorm{y}}\,\dy\\
&\leq \Cc{c} \int_0^R R^{-n-\snorm{\alpha}}\, \snorm{\alpha}^{\frac{n-2}{4}}\, r^{\frac{n}{2}}\, \snormL{\hankelone{\frac{n-2}{2}}\bp{\sqrt{-\alpha}\cdot r}} \, \e^{\frac{\snorm{\rey}}{2} r}\,\dr\\
&\leq \Cc{c} \int_0^\infty R^{-n-\snorm{\alpha}}\, \snorm{\alpha}^{-1}\, s^{\frac{n}{2}}\, \snormL{\hankelone{\frac{n-2}{2}}\bp{\sqrt{-\alpha}\cdot\snorm{\alpha}^{-\half} s}}\, \e^{\frac{\snorm{\rey}}{2} \snorm{\alpha}^{-1/2} s} \,\ds.
\end{align*}
Employing estimate \eqref{estHankelFktAtZero} in combination with \eqref{estHankelFktAtInfty}, in the case $n>2$ we obtain
\begin{align*}
\snorml{D_x^\alpha I_1(x)} 
&\leq \Cc{c}\,  R^{-n-\snorm{\alpha}}\, \snorm{\alpha}^{-1}\, \Bp{ \int_0^1 s \e^{\frac{\snorm{\rey}}{2} \snorm{\alpha}^{-1/2} s}\,\ds + 
\int_1^\infty s^{\frac{n-1}{2}} \e^{\bb{-\Im (\sqrt{-\alpha})+\frac{\snorm{\rey}}{2}}\snorm{\alpha}^{-1/2} s}\,\ds } \\
&\leq \Cc{c}\,  R^{-n-\snorm{\alpha}}\, \snorm{k}^{-1}\, \Bp{ \int_0^1 s \e^{s}\,\ds + 
\int_1^\infty s^{\frac{n}{2}}\, s^{-\half}\, \e^{-\const{Imalpha} \sqrt{\frac{\per}{2\pi}} s}\,\ds } \\
&\leq \Cc{c}\,  R^{-n-\snorm{\alpha}}\, \snorm{k}^{-1}. 
\end{align*}
When $n=2$, we use estimate \eqref{estHankelzeroFktAtZero} in combination with \eqref{estHankelFktAtInfty} and also obtain
\begin{align*}
\snorml{D_x^\alpha I_1(x)}
&\leq \Cc{c}  R^{-2-\snorm{\alpha}} \snorm{\alpha}^{-1} \Bp{ \int_0^1 s \snorml{\log (s)} \e^{\frac{\snorm{\rey}}{2} \snorm{\alpha}^{-1/2} s}\ds 
+\! \! \int_1^\infty \! s^{\half}\e^{\bb{-\Im (\sqrt{-\alpha})+\frac{\rey}{2}}\snorm{\alpha}^{-1/2} s}\ds \! } \\
&\leq \Cc{c}\,  R^{-2-\snorm{\alpha}}\, \snorm{k}^{-1}\, \Bp{ \int_0^1 s \snorml{\log (s)} \, \e^{s}\,\ds 
+ \int_1^\infty s^{\half}\,\e^{-\const{Imalpha} \sqrt{\frac{\per}{2\pi}} s}\,\ds } \\
&\leq \Cc{c}\,  R^{-2-\snorm{\alpha}}\, \snorm{k}^{-1}
\end{align*}
in this case.
In order to estimate $D_x^\alpha I_2$, we again utilize \eqref{EstDerivativeGammaH}, which directly yields
\[
\snorm{D_x^\alpha I_2(x)}
\leq \Cc{c} \int_{B^{3R}} \snorm{x-y}^{2-n-\snorm{\alpha}}\, \snorm{k}^{-1}\,\snorm{y}^{-(n+1)}\,\dy
\leq \Cc{c}\, \snorm{k}^{-1}\,R^{-n-\snorm{\alpha}}.
\]
For the estimate of $D_x^\alpha I_3$, at first we use \eqref{EstDerivativeGammaH} with $\snorm{\alpha}=0$ and obtain
\begin{align*}
\snorml{I_3(x)} 
&\leq \Cc{c} \int_{B_{4R,R/2}} \snorm{x-y}^{1-n}\, \snorm{k}^{-1}\,\snorm{y}^{-(n+1)}\,\dy \\
& \leq \Cc{c} \snorm{k}^{-1}\,R^{-(n+1)} \int_{B_{4R,R/2}} \snorm{x-y}^{1-n}\, \dy 
\leq \Cc{c}\, \snorm{k}^{-1}\,R^{-n}.
\end{align*}
In the case $\snorm{\alpha} \neq 0$, we can use integration by parts to deduce
\[
D_x^\alpha I_3(x) = (-1)^{\snorm{\alpha}} \int_{B_{4R,R/2}} \partial_i\fundsollaplace(x-y)\,\partial_j D_x^\alpha \bb{\fundsolssrkrey \cdot \cutoff_R}(y) \, \dy.
\]
Now \eqref{EstDerivativeGammaH} leads to the estimate
\begin{align*}
\snorm{D_x^\alpha I_3(x)}
&\leq \Cc{c} \sum_{\beta \leq \alpha+ e_j} \ \int_{B_{4R,R/2}} \snorm{x-y}^{1-n}\, \snorm{k}^{-1}\,\snorm{y}^{2-n-\snorm{\beta}} R^{-\snorm{\alpha - \beta}-1}\,\dy \\
& \leq \Cc{c} \snorm{k}^{-1}\,R^{-n-\snorm{\alpha}-1} \int_{B_{4R,R/2}} \snorm{x-y}^{1-n}\, \dy 
\leq \Cc{c}\, \snorm{k}^{-1}\,R^{-n-\snorm{\alpha}}.
\end{align*}
Since $\snorm{x}=2R$, we conclude \eqref{pointwiseEstLemEst1} by collecting the estimates for $D_x^\alpha I_1$, $D_x^\alpha I_2$ and $D_x^\alpha I_3$.
\end{proof}

The preparations above enable us to
establish pointwise estimates of the 
purely periodic part of the fundamental solutions to the time-periodic 
Stokes and Oseen problem. We can thereby finalize the proof of Theorem \ref{MainThm}.

\begin{proof}[Proof of Theorem \ref{MainThm}]
At first we remark that the function
\begin{align}\label{EqDefMultiplier}
M \colon \dualgrp \to \C, \quad M(k, \xi) \coloneqq \frac{1-\delta_\Z(k)}{\snorm{\xi}^2 + i\bp{\perf k + \rey \xi_1}} 
\end{align}
is an element of $\LR{\infty}(\grp)$. Thus \eqref{MainThm_deffundsolcompl} is a well-defined object in $\TDR(G)^{n \times n}$. Since it holds $\FT_{\R^n}[\fundsolpres]= -i \frac{\xi}{\snorm{\xi}^2}$, we obtain
\[
\FT_G[ \grad \fundsoltppres] = \FT_{\R^n} [ \grad \fundsolpres] \otimes \FT_\torus[\delta_\torus] = \frac{\xi \otimes \xi}{\snorm{\xi}^2} \cdot 1_\Z.
\]
Because we also have $\bp{\snorm{\xi}^2 + i \rey \xi_1} \cdot \FT_{\R^n}[\fundsolvel] = \bp{\idmatrix - \frac{\xi \otimes \xi}{\snorm{\xi}^2}}$, we further deduce
\[
\Bp{\snorm{\xi}^2+ i \bp{\perf k + \rey \xi_1}} \cdot \FT_G[\fundsolvel \otimes 1_\torus] = \Bp{\idmatrix - \frac{\xi \otimes \xi}{\snorm{\xi}^2}} \delta_\Z(k),
\]
which finally leads us to
\[
\Bp{\snorm{\xi}^2+ i \bp{\perf k + \rey \xi_1}} \cdot \FT_G[\fundsoltpvel ] + \FT_G [\grad \fundsoltppres] = \idmatrix.
\]
Applying the inverse Fourier transform to this equality, we conclude that $(\fundsoltpvel, \fundsoltppres)$ is, in fact, a fundamental solution to \eqref{tpFundsolEq} since clearly $\Div \fundsolvel= \Div \fundsolcompl=0$.

We continue with the derivation of \eqref{MainThm_ComplPointwiseEst}, for which we will utilize Lemma \ref{pointwiseEstLem}. 
The decay rate established for $\fundsolssrkrey$ in Lemma \ref{lemfunsolssrkreyDistr} implies that 
$\fundsollaplace \ast \fundsolssrkrey$ (defined in  \eqref{defoffundsollaplaceCONVLfundsolssrk}) is a tempered distribution on $\R^n$. Therefore, we may apply the Fourier transform to the derivatives of this distribution. Then identity \eqref{fundsolssrkreyMultiplier} yields
\[
\FT_{\R^n} \bb{\partial_j \partial_l [\fundsollaplace \ast \fundsolssrkrey]}(\xi) 
= \frac{\xi_j \xi_l}{\snorm{\xi}^2} \frac{1}{\snorm{\xi}^2 + i \bp{\perf k+\rey \xi_1}}
\]
and, in particular, 
\[
\FT_{\R^n} \bb{\Delta [\fundsollaplace \ast \fundsolssrkrey]}(\xi) 
= \frac{1}{\snorm{\xi}^2 + i \bp{\perf k+\rey \xi_1}}.
\]
Hence definition \eqref{MainThm_deffundsolcompl} yields
\[
D_x^\alpha \fundsolcompl_{jl} = \iFT_\torus \bb{ \bp{1-\delta_\Z(k)} \bb{\delta_{jl}\Delta - \partial_j \partial_l } D_x^\alpha [\fundsollaplace \ast \fundsolssrkrey]}.
\]
We may now utilize Lemma \ref{pointwiseEstLem}. For $r \in [2, \infty)$ with H\"older conjugate $r'$, the Hausdorff-Young inequality in combination with estimate \eqref{pointwiseEstLemEst1} yields
\begin{align*}
\norm{\fundsolcompl_{jl}(\cdot, x)}_{\LR{r}(\torus)} & \leq \Bp{\sum_{k \in \Z} \snormL{\bp{1-\delta_\Z(k)} \cdot \bb{\delta_{jl} \Delta - \partial_j \partial_l} D_x^\alpha[\fundsollaplace \ast \fundsolssrkrey](x)}^{r'}}^\frac{1}{r'} \\
& \leq \Cc{c} \snorm{x}^{-n-\snorm{\alpha}} \Bp{\sum_{k \in \Z \setminus \set{0}} \snorm{k}^{-r'}}^\frac{1}{r'} \leq \Cc{c} \snorm{x}^{-n-\snorm{\alpha}},
\end{align*}
which finally implies \eqref{MainThm_ComplPointwiseEst}.

Now we go on with the derivation of \eqref{MainThm_ComplSummability}, for which we generalize the approach from \cite[Proof of Lemma 4.5]{KyedGaldi_TPSolNS3d} to arbitrary dimension $n\geq 2$.
Equation \eqref{MainThm_deffundsolcompl} leads to the representation
\[
\fundsolcompl_{jl}
= \bb{\delta_{jl} \bp{\riesztrans_h \riesztrans_h} - \riesztrans_j \riesztrans_l } \circ \iFT_G \Bb{ M_0 \cdot \FT_G \bb{\iFT_G (\calk)} },
\]
where $\riesztrans_j$ denotes the Riesz transform 
\[
\riesztrans_j \colon \SR(G) \to \TDR(G), \quad \riesztrans_j(f) \coloneqq \iFT_{\R^n} \Bb{ \frac{\xi_j}{\snorm{\xi}} \FT_{\R^n}(f) } 
\] 
and
\[
M_0 \colon \dualgrp \to \C, \quad 
M_0(k, \xi) \coloneqq \frac{\bp{1-\delta_\Z(k)} \snorm{k}^{\frac{2}{n+2}}\bp{1+\snorm{\xi}^2}^\frac{n}{n+2}}{\snorm{\xi}^2 + i \bp{\perf k + \rey \xi_1} }
\]
as well as
\[
\calk \colon \dualgrp \to \C, \quad
\calk(k, \xi) \coloneqq \bp{1-\delta_\Z(k)} \snorm{k}^{-\frac{2}{n+2}} \bp{1+ \snorm{\xi}^2}^{-\frac{n}{n+2}}.
\]
It is well known that $\riesztrans_j$ extends to a bounded linear operator $\LR{q}(G) \to \LR{q}(G)$ for all $q \in (1, \infty)$; see for example \cite[Corollary 4.2.8]{Grafakos1}. 

In order to obtain that $M_0$ is an $\LR{q}(G)$ multiplier,
we adapt the method applied in \cite[Proof of Theorem 4.8]{mrtpns}. To do so, we consider a function $\cutoff \in \CRci(\R;\R)$ with 
\begin{align*}
\cutoff(\eta)=
\begin{cases}
0 & \text{for } \snorm{\eta} \leq \half,\\
1 & \text{for } \snorm{\eta} \geq 1,
\end{cases}
\end{align*}
and define 
\[
m_0 \colon \R \times \R^n \to \C, \quad 
m_0(\eta, \xi) \coloneqq \frac{\chi(\eta) \snorm{\eta}^{\frac{2}{n+2}}\bp{1+\snorm{\xi}^2}^\frac{n}{n+2}}{\snorm{\xi}^2 + i \bp{\perf \eta + \rey \xi_1} },
\]
and the natural embedding
\[
\pi \colon \dualgrp \to \R \times \R^n, \quad \pi(k,\xi)\coloneqq \bp{k, \xi}.
\]
We remark that $m_0$ is a continuous and bounded function since the numerator vanishes in a neighborhood of $0$, which is the only zero of the denominator. With the help of Marcinkiewicz's Multiplier Theorem (see for instance \cite[Corollary 5.2.5]{Grafakos1}), one readily verifies that $m_0$ is an $\LR{q}(\R \times \R^n)$ multiplier for all $q \in (1, \infty)$. Additionally, we have $M_0 = m_0 \circ \pi$. Thus an application of the Transference Principle, see \cite[Theorem B.2.1]{EdwardsGaudryBook}, implies that $M_0$ is an $\LR{q}(G)$ multiplier for all $q \in (1, \infty)$. Hence we obtain \eqref{MainThm_deffundsolcompl} if we can show $\iFT_\grp(\calk) \in \LR{r}(G)$.

If we identify $\torus$ with the interval $(- \half \per, \half \per ]$, we obtain
\[
\iFT_{\torus} \Bb{ \bp{1-\delta_\Z(k)} \snorm{k}^{-\frac{2}{n+2}}} (t) = \Cc{c}\snorm{t}^{-\frac{n}{n+2}} + h(t)
\]
for some function $h \in \CR{\infty}(\R / \per \Z)$; see for instance \cite[Example 3.1.19]{Grafakos1}. Furthermore, one can derive the estimate
\[
\snormL{ \iFT_{\R^n} \Bb{ \bp{1+\snorm{\xi}^2}^{-\frac{n}{n+2}}}(x)} 
\leq \Cc{c} \Bp{ \snorm{x}^{-\frac{n^2}{n+2}} \chi_{B_2}(x) + \e^{-\frac{\snorm{x}}{2}}};
\]
see for example \cite[Proposition 6.1.5]{Grafakos2}. Therefore, we conclude
\[
\iFT_G(\calk) = \iFT_{\torus} \Bb{ \bp{1-\delta_\Z(k)} \snorm{k}^{-\frac{2}{n+2}}} 
\otimes \iFT_{\R^n} \Bb{ \bp{1+\snorm{\xi}^2}^{-\frac{n}{n+2}}} \in \LR{r}(G)
\]
for all $r \in \bp{1, \frac{n+2}{n}}$, and we have verified \eqref{MainThm_ComplSummability}. In order to show \eqref{MainThm_ComplSummabilityGradient}, we proceed in a similar way. We consider the identity
\[
\partial_m \fundsolcompl_{jl}
= \bb{\delta_{jl} \bp{\riesztrans_h \riesztrans_h} - \riesztrans_j \riesztrans_l } \circ \iFT_G \Bb{ M_m \cdot \FT_G \bb{\iFT_G (\calj)} },
\]
where
\[
M_m \colon \dualgrp \to \C, \quad 
M_m(k, \xi) \coloneqq \frac{\bp{1-\delta_\Z(k)} \snorm{k}^{\frac{1}{n+2}}\bp{1+\snorm{\xi}^2}^\frac{n}{2(n+2)} \, i\xi_m}{\snorm{\xi}^2 + i \bp{\perf k + \rey \xi_1} }
\]
and
\[
\calj \colon \dualgrp \to \C, \quad
\calj(k, \xi) \coloneqq \bp{1-\delta_\Z(k)} \snorm{k}^{-\frac{1}{n+2}} \bp{1+ \snorm{\xi}^2}^{-\frac{n}{2(n+2)}}.
\]
With the same arguments as above, we conclude $\partial_m \fundsolcompl \in \LR{r}(G)$ for all $r \in \bp{1, \frac{n+2}{n+1}}$.
In particular, this yields $\partial_m \fundsolcompl \in \LRloc{1}(G)$, which finally leads to $\partial_m \fundsolcompl \in \LR{1}(G)$ by the asymptotic behavior from \eqref{MainThm_ComplPointwiseEst} for $\snorm{\alpha} = 1$. Consequently, we have also shown \eqref{MainThm_ComplSummabilityGradient}.

The convolution $\fundsolcompl*f$ can be expressed in terms of a Fourier multiplier
\begin{align*}
\fundsolcompl*f = \iFT_\grp\Bb{\Mmultiplier(k,\xi)\Bp{\idmatrix-\frac{\xi\otimes\xi}{\snorm{\xi}^2}} \FT_\grp\nb{f}},
\end{align*}
with $M$ given by \eqref{EqDefMultiplier}. As already mentioned, $M\in\LR{\infty}(\dualgrp)$. As one may verify, also the functions $(k,\xi) \mapsto -\xi_j \xi_l \cdot M(k,\xi)$ and $(k,\xi) \mapsto ik \cdot M(k,\xi)$ are bounded.
Based on this information, \eqref{MainThm_ConvlFundsolcomplLqEst} can be established from the theory on Fourier multipliers in a similar way as above.
\end{proof}

\bibliographystyle{abbrv}

\end{document}